\documentclass[10pt]{amsart}
\usepackage{amssymb,MnSymbol}
\usepackage{amsthm,amsmath}
\usepackage{epic,cite}

\title[Characterizations of idempotent discrete uninorms]{Characterizations of idempotent discrete uninorms}

\author{Miguel Couceiro}
\address{LORIA, CNRS - Inria Nancy Grand Est - Universit\'e de Lorraine, BP239, 54506 Vandoeuvre-l\`es-Nancy, France} \email{miguel.couceiro[at]inria.fr }

\author{Jimmy Devillet}
\address{Mathematics Research Unit, University of Luxembourg, Maison du Nombre, 6, avenue de la Fonte, L-4364 Esch-sur-Alzette, Luxembourg}
\email{jimmy.devillet[at]uni.lu}

\author{Jean-Luc Marichal}
\address{Mathematics Research Unit, University of Luxembourg, Maison du Nombre, 6, avenue de la Fonte, L-4364 Esch-sur-Alzette, Luxembourg}
\email{jean-luc.marichal[at]uni.lu}

\date{June 27, 2017}

\theoremstyle{plain}
\newtheorem{theorem}{Theorem}
\newtheorem{lemma}[theorem]{Lemma}
\newtheorem{proposition}[theorem]{Proposition}
\newtheorem{corollary}[theorem]{Corollary}

\theoremstyle{definition}
\newtheorem{definition}[theorem]{Definition}
\newtheorem{example}[theorem]{Example}

\theoremstyle{remark}

\newtheorem{remark}{Remark}

\begin{document}
\begin{abstract}
In this paper we provide an axiomatic characterization of the idempotent discrete uninorms by means of three conditions only: conservativeness, symmetry, and nondecreasing monotonicity. We also provide an alternative characterization involving the bisymmetry property. Finally, we provide a graphical characterization of these operations in terms of their contour plots, and we mention a few open questions for further research.
\end{abstract}



\maketitle

\setlength{\unitlength}{8ex}

\section{Introduction}

Aggregation functions defined on linguistic scales (i.e., finite chains) have been intensively investigated for about two decades; see, e.g., \cite{DeBFodRuiTor09,DeBMes03,Fod00,LiLiuFod15,MasMayTor99,MasMonTor03,MasMonTor04,MasMonTor11,MaySunTor05,MayTor05,RuiTor15,SuLiu17}. Among these functions, discrete fuzzy connectives (such as discrete uninorms) are binary operations that play an important role in fuzzy logic.

This short paper focuses on characterizations of the class of idempotent discrete uninorms. Recall that a discrete uninorm is a binary operation on a finite chain that is associative, symmetric, nondecreasing (in each variable), and has a neutral element.

A first characterization of the class of idempotent discrete uninorms was given by De Baets et al.~\cite{DeBFodRuiTor09}. This characterization reveals that any idempotent discrete uninorm is a combination of the minimum and maximum operations. In particular, such an operation is \emph{conservative} in the sense that it always outputs one of the input values.

The outline of this paper is as follows. After presenting some preliminary results on conservative operations in Section 2, we show in Section 3 that the idempotent discrete uninorms are exactly those operations that are conservative, symmetric, and nondecreasing (Theorem~\ref{thm:main}). This new axiomatic characterization is rather surprising since it requires neither associativity nor the existence of a neutral element. We also present a graphical characterization of these operations in terms of their contour plots (Theorem~\ref{thm:GC}) as well as an algebraic translation of this characterization (Theorem~\ref{thm:QoB}). These characterizations show us a very easy way to generate all the possible idempotent discrete uninorms on a given finite chain. In Section 4 we provide an alternative axiomatic characterization of this class in terms of the bisymmetry property. Specifically, we show that the idempotent discrete uninorms are exactly those operations that are idempotent, bisymmetric, nondecreasing, and have neutral elements. More generally, we also show that the whole class of discrete uninorms can also be axiomatized by simply suppressing idempotency in the latter characterization and that this result also holds on arbitrary chains (Theorem~\ref{thm:mainb}). Finally, Section 5 is devoted to some concluding remarks and open questions.

This paper is an extended version of a conference paper that appeared in \cite{CouDevMar18}.

\section{Preliminaries}

In this section we present some basic definitions and preliminary results.

Let $X$ be an arbitrary nonempty set and let $\Delta_X=\{(x,x)\mid x\in X\}$.

\begin{definition}
An operation $F\colon X^2\to X$ is said to be
\begin{itemize}
\item \emph{idempotent} if $F(x,x)=x$ for all $x\in X$.
\item \emph{conservative} (or \emph{selective}) if $F(x,y)\in\{x,y\}$ for all $x,y\in X$.
\item \emph{associative} if $F(F(x,y),z)=F(x,F(y,z))$ for all $x,y,z\in X$.
\end{itemize}
\end{definition}

\begin{remark}\label{rem:ConsJ}
Conservative operations were introduced first in \cite{PouRosSto96} and then independently as \emph{locally internal} operations in \cite{MarMayTor03}. By definition, the output value of such an operation must always be one of the input values. In particular, any conservative operation is idempotent. Moreover, such an operation $F\colon X^2\to X$ can be ``discreticized'' in the sense that, for any nonempty subset $S$ of $X$, the restriction of $F$ to $S^2$ ranges in $S$. More precisely, it can be shown \cite{CouMar10} that the following conditions are equivalent:
\begin{enumerate}
\item[(i)] $F$ is conservative.
\item[(ii)] For any $\varnothing\neq S\subseteq X$, we have $F(S^2)\subseteq S$.
\item[(iii)] For any $\varnothing\neq S\subseteq X$ and any $x,y\in S$, if $F(x,y)\in S$ then $x\in S$ or $y\in S$.
\end{enumerate}
\end{remark}

\begin{definition}
Let $F\colon X^2\to X$ be an operation.
\begin{itemize}
\item An element $e\in X$ is said to be a \emph{neutral element} of $F$ (or simply a \emph{neutral element}) if $F(x,e)=F(e,x)=x$ for all $x\in X$. In this case we easily show by contradiction that such a neutral element is unique.
\item The points $(x,y)$ and $(u,v)$ of $X^2$ are said to be \emph{connected for $F$} (or simply \emph{connected}) if $F(x,y)=F(u,v)$. We observe that ``being connected'' is an equivalence relation. The point $(x,y)$ of $X^2$ is said to be \emph{isolated for $F$} (or simply \emph{isolated}) if it is not connected to another point in $X^2$.
\end{itemize}
\end{definition}

\begin{proposition}\label{prop:IdIs}
Let $F\colon X^2\to X$ be an idempotent operation. If the point $(x,y)\in X^2$ is isolated, then it lies on $\Delta_X$, that is, $x=y$.
\end{proposition}

\begin{proof}
Let $(x,y)\in X^2$ be an isolated point. By idempotency we immediately have $F(x,y)=F(F(x,y),F(x,y))$. But since $(x,y)$ is isolated we necessarily have $(x,y)=(F(x,y),F(x,y))$, and therefore $x=y$.
\end{proof}

\begin{remark}
We observe that idempotency is necessary in Proposition~\ref{prop:IdIs}. Indeed, consider the operation $F\colon X^2\to X$, where $X=\{a,b\}$, defined as $F(x,y)=a$, if $(x,y)=(a,b)$, and $F(x,y)=b$, otherwise. Then $(a,b)$ is isolated and $a\neq b$. The contour plot of $F$ is represented in Figure~\ref{fig:1}. Here and throughout, connected points are joined by edges. To keep the figures simple we sometimes omit the edges obtained by transitivity.
\end{remark}

\begin{figure}[tbp]
\begin{center}
\begin{scriptsize}
\begin{picture}(2,2)
\multiput(0.5,0.5)(0,1){2}{\multiput(0,0)(1,0){2}{\circle*{0.12}}}%
\put(0.5,0.4){\makebox(0,0)[t]{$(a,a)$}}
\put(1.5,0.4){\makebox(0,0)[t]{$(b,a)$}}
\put(0.5,1.4){\makebox(0,0)[t]{$(a,b)$}}
\put(1.5,1.4){\makebox(0,0)[t]{$(b,b)$}}
\drawline[1](0.5,0.5)(1.5,0.5)(1.5,1.5)
\end{picture}
\end{scriptsize}
\caption{A non-idempotent operation}
\label{fig:1}
\end{center}
\end{figure}
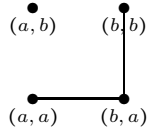

We note that any conservative operation $F\colon X^2\to X$ has at most one isolated point. This observation immediately follows from Proposition~\ref{prop:IdIs} and the fact that $F(x,y)\in\{x,y\}=\{F(x,x),F(y,y)\}$ for all $x,y\in X$.

Some conservative operations have neutral elements (e.g., $F(x,y)=\max\{x,y\}$ on $X=\{1,2,3\}$ has the neutral element $e=1$) and some have not (e.g., $F(x,y)=x$). The following lemma provides an easy graphical test for checking whether a conservative operation has a neutral element.

\begin{proposition}\label{prop:ee}
Let $F\colon X^2\to X$ be a conservative operation and let $e\in X$. Then $e$ is a neutral element if and only if $(e,e)$ is isolated.
\end{proposition}

\begin{proof}
(Necessity) If $(e,e)$ is not isolated, then there exists $(x,y)\neq (e,e)$ such that $e=F(e,e)=F(x,y)\in\{x,y\}$ and hence $x=e$ or $y=e$. If $x=e$, then $y\neq e$ and $e=F(e,y)=y$, a contradiction. We arrive at a similar contradiction when $y=e$.

(Sufficiency) If $e$ is not a neutral element, then there exists $u\in X\setminus\{e\}$ such that $F(u,e)=e=F(e,e)$ or $F(e,u)=e=F(e,e)$. In both cases, $(e,e)$ is not isolated, a contradiction.
\end{proof}

\begin{corollary}
Any isolated point $(x,y)$ of a conservative operation $F\colon X^2\to X$ is unique and lies on $\Delta_X$. Moreover, the element $x{\,}(=y)$ is the neutral element of $F$.
\end{corollary}

\begin{remark}
Proposition~\ref{prop:ee} no longer holds if conservativeness is relaxed into idempotency. Indeed, by simply taking $X=\{a,b,c\}$ we can easily construct an idempotent operation with an isolated point on $\Delta_X$ and no neutral element (see Figure~\ref{fig:2}). Also, it is easy to construct an idempotent operation with a neutral element and no isolated point (see Figure~\ref{fig:3}). It is also noteworthy that there are idempotent operations with more than one isolated point (see Figure~\ref{fig:4}).
\end{remark}

\begin{figure}[tbp]
\begin{center}
\begin{scriptsize}
\begin{picture}(3,3)
\multiput(0.5,0.5)(0,1){3}{\multiput(0,0)(1,0){3}{\circle*{0.12}}}%
\put(0.5,0.4){\makebox(0,0)[t]{$(a,a)$}}
\put(1.5,0.4){\makebox(0,0)[t]{$(b,a)$}}
\put(2.5,0.4){\makebox(0,0)[t]{$(c,a)$}}
\put(0.5,1.4){\makebox(0,0)[t]{$(a,b)$}}
\put(1.5,1.4){\makebox(0,0)[t]{$(b,b)$}}
\put(2.5,1.4){\makebox(0,0)[t]{$(c,b)$}}
\put(0.5,2.4){\makebox(0,0)[t]{$(a,c)$}}
\put(1.5,2.4){\makebox(0,0)[t]{$(b,c)$}}
\put(2.5,2.4){\makebox(0,0)[t]{$(c,c)$}}
\drawline[1](2.5,1.5)(2.5,2.5)(1.5,2.5)
\drawline[1](2.5,0.5)(1.5,0.5)(1.5,1.5)(0.5,1.5)(0.5,2.5)
\end{picture}
\end{scriptsize}
\caption{An operation with no neutral element}
\label{fig:2}
\end{center}
\end{figure}
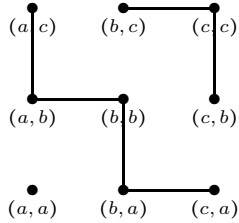

\begin{figure}[tbp]
\begin{minipage}{0.49\textwidth}
\begin{center}
\begin{scriptsize}
\begin{picture}(3,3)
\multiput(0.5,0.5)(0,1){3}{\multiput(0,0)(1,0){3}{\circle*{0.12}}}%
\put(0.5,0.4){\makebox(0,0)[t]{$(a,a)$}}
\put(1.5,0.4){\makebox(0,0)[t]{$(b,a)$}}
\put(2.5,0.4){\makebox(0,0)[t]{$(c,a)$}}
\put(0.5,1.4){\makebox(0,0)[t]{$(a,b)$}}
\put(1.5,1.4){\makebox(0,0)[t]{$(b,b)$}}
\put(2.5,1.4){\makebox(0,0)[t]{$(c,b)$}}
\put(0.5,2.4){\makebox(0,0)[t]{$(a,c)$}}
\put(1.5,2.4){\makebox(0,0)[t]{$(b,c)$}}
\put(2.5,2.4){\makebox(0,0)[t]{$(c,c)$}}
\drawline[1](2.5,1.5)(2.5,2.5)(1.5,2.5)
\drawline[1](1.5,0.5)(0.5,0.5)(0.5,1.5)
\drawline[1](2.5,0.5)(0.5,2.5)
\end{picture}
\end{scriptsize}
\caption{An operation with no isolated point}
\label{fig:3}
\end{center}
\end{minipage}
\hfill
\begin{minipage}{0.49\textwidth}
\begin{center}
\begin{scriptsize}
\begin{picture}(3,3)
\multiput(0.5,0.5)(0,1){3}{\multiput(0,0)(1,0){3}{\circle*{0.12}}}%
\put(0.5,0.4){\makebox(0,0)[t]{$(a,a)$}}
\put(1.5,0.4){\makebox(0,0)[t]{$(b,a)$}}
\put(2.5,0.4){\makebox(0,0)[t]{$(c,a)$}}
\put(0.5,1.4){\makebox(0,0)[t]{$(a,b)$}}
\put(1.5,1.4){\makebox(0,0)[t]{$(b,b)$}}
\put(2.5,1.4){\makebox(0,0)[t]{$(c,b)$}}
\put(0.5,2.4){\makebox(0,0)[t]{$(a,c)$}}
\put(1.5,2.4){\makebox(0,0)[t]{$(b,c)$}}
\put(2.5,2.4){\makebox(0,0)[t]{$(c,c)$}}
\drawline[1](1.5,2.5)(0.5,2.5)(0.5,1.5)(2.5,1.5)(2.5,0.5)(1.5,0.5)(1.5,2.5)
\end{picture}
\end{scriptsize}
\caption{An operation with two isolated points}
\label{fig:4}
\end{center}
\end{minipage}
\end{figure}

\begin{proposition}\label{prop:Tcons}
An operation $F\colon X^2\to X$ is conservative if and only if it is idempotent and every point $(x,y)\in X^2\setminus\Delta_X$ is connected to either $(x,x)$ or $(y,y)$.
\end{proposition}

\begin{proof}
Clearly, $F$ is conservative if and only if it is idempotent and for every distinct $x,y\in X$ we have either $F(x,y)=x=F(x,x)$ or $F(x,y)=y=F(y,y)$.
\end{proof}

\begin{remark}
Proposition~\ref{prop:Tcons} provides an easy graphical test for checking whether an idempotent operation is conservative. For instance, none of the idempotent operations represented in Figures~\ref{fig:2}--\ref{fig:4} is conservative because in each case the point $(a,c)$ is connected to neither $(a,a)$ nor $(c,c)$.
\end{remark}

\begin{proposition}\label{prop:Te3}
An operation $F\colon X^2\to X$ has a neutral element if and only if there are a vertical section and a horizonal section of $X^2$ that intersect on $\Delta_X$ and such that the restriction of $F$ to each of these sections is the identity function.
\end{proposition}

\begin{proof}
The result immediately follows from the definition of a neutral element.
\end{proof}

\begin{remark}
Proposition~\ref{prop:Te3} provides a graphical test for checking the existence of a neutral element. For instance, we can easily see that the operation represented in Figure~\ref{fig:3} has $b$ as the neutral element. Note that when the operation is conservative, by Proposition~\ref{prop:ee} it suffices to search for an isolated point on $\Delta_X$.
\end{remark}

\section{Main results}

We now focus on characterizations of the class of idempotent discrete uninorms. These operations are defined on finite chains. Without loss of generality we will only consider the $n$-element chains $L_n=\{1,\ldots,n\}$, $n\geq 1$, endowed with the usual ordering relation $\leq$.

Recall that an operation $F\colon L_n^2\to L_n$ is said to be \emph{nondecreasing in each variable} (or simply \emph{nondecreasing}) if $F(x,y)\leq F(x',y')$ whenever $x\leq x'$ and $y\leq y'$.

\begin{definition}[see, e.g., \cite{DeBFodRuiTor09}]
A \emph{discrete uninorm} on $L_n$ is an operation $U\colon L_n^2\to L_n$ that is associative, symmetric, nondecreasing, and has a neutral element.
\end{definition}

A first characterization of the class of idempotent discrete uninorms was established by De Baets et al.~\cite{DeBFodRuiTor09}. We state this result in the following theorem. Although this characterization is somewhat intricate, it shows that any idempotent discrete uninorm is conservative.

\begin{theorem}[{see \cite[Theorem~3]{DeBFodRuiTor09}}]\label{thm:Deb1}
An operation $F\colon L_n^2\to L_n$ with a neutral element $e$ is an idempotent discrete uninorm if and only if there exists a nonincreasing map $g\colon\{1,\ldots,e\}\to\{e,\ldots,n\}$ (nonincreasing means that $g(x)\geq g(y)$ whenever $x\leq y$), with $g(e)=e$, such that
$$
F(x,y) ~=~
\begin{cases}
\min\{x,y\}, & \text{if $y\leq\overline{g}(x)$ and $x\leq\overline{g}(1)$},\\
\max\{x,y\}, & \text{otherwise},
\end{cases}
$$
where $\overline{g}\colon L_n\to L_n$ is defined by
$$
\overline{g}(x) ~=~
\begin{cases}
g(x), & \text{if $x\leq e$},\\
\max\{z\in\{1,\ldots,e\}\mid g(z)\geq x\}, & \text{if $e\leq x\leq g(1)$},\\
1, & \text{if $x>g(1)$}.
\end{cases}
$$
\end{theorem}

\begin{remark}\label{rem:ccons}
The fact that any idempotent discrete uninorm is conservative can also be easily proved by following the first few steps of the proof of \cite[Theorem~3]{CzoDre84}, which actually hold on an arbitrary chain (i.e., totally ordered sets).
\end{remark}

We now show that the idempotent discrete uninorms are exactly those operations that are conservative, symmetric, and nondecreasing (see Theorem~\ref{thm:main}).

First consider the following lemma, which actually holds on arbitrary, not necessarily finite, chains.

\begin{lemma}\label{lemma:prel34}
If $F\colon L_n^2\to L_n$ is idempotent, nondecreasing, and has a neutral element $e\in L_n$, then $F|_{\{1,\ldots,e\}^2}=\min$ and $F|_{\{e,\ldots,n\}^2}=\max$.
\end{lemma}

\begin{proof}
For any $x,y\in\{1,\ldots,e\}$ such that $x\leq y$, we have $x=F(x,x)\leq F(x,y)\leq F(x,e)=x$ and $x=F(x,x)\leq F(y,x)\leq F(e,x)=x$. This shows that $F|_{\{1,\ldots,e\}^2}=\min$. The other identity can be proved similarly.
\end{proof}

\begin{proposition}\label{prop:main3}
If $F\colon L_n^2\to L_n$ is conservative, symmetric, and nondecreasing, then it is associative and it has a neutral element.
\end{proposition}

\begin{proof}
Let us first prove that $F$ has a neutral element. We proceed by induction on the size $n$ of the chain. There is nothing to prove if $n=1$. We can easily see by inspection that there are only two possible operations if $n=2$ and four possible operations if $n=3$. The contour plot of these operations are given in Figures~\ref{fig:5} and \ref{fig:6}, respectively. Now suppose that the result holds for any $(n-1)$-element chain and consider an operation $F\colon L_n^2\to L_n$ that is conservative, symmetric, and nondecreasing. By conservativeness and symmetry we then have $F(1,n)=F(n,1)\in\{1,n\}$. We may suppose that $F(1,n)=F(n,1)=1$; the other case can be dealt with dually. By nondecreasing monotonicity, we also have $F(1,x)=F(x,1)=1$ for all $x\in L_n$. Consider the subchain $L'=L_n\setminus\{1\}$. Clearly, the operation $F'=F|_{L'}$ is conservative, symmetric, and nondecreasing. By the induction hypothesis, $F'$ has a neutral element $e\in L'$. Let us show that $e$ is also a neutral element of $F$. Suppose that this is not true. Then, by Proposition~\ref{prop:ee} the point $(e,e)$ is isolated for $F'$ but not for $F$. This means that there exists $x\in L_n$ such that $1=F(1,x)=F(x,1)=F(e,e)=e$, which contradicts the fact that $e\in L'$.

Now, let $F\colon L_n^2\to L_n$ be an operation that is conservative, symmetric, and nondecreasing. We just showed that $F$ must have a neutral element $e\in L_n$. To see that $F$ is associative, let $x,y,z\in L_n$ be arbitrary and let us show that the identity $F(F(x,y),z)=F(x,F(y,z))$ holds. Assume that $x\leq y\leq z$ (the other five permutations can be treated similarly). We have three cases to examine:
\begin{itemize}
\item Suppose $x\leq y\leq z\leq e$ or $e\leq x\leq y\leq z$. Then the result immediately follows from Lemma~\ref{lemma:prel34}.
\item Suppose $x\leq y\leq e\leq z$. By Lemma~\ref{lemma:prel34}, we have $F(x,y)=\min\{x,y\}=x$.
    \begin{itemize}
    \item If $F(x,z)=x$, then $F(F(x,y),z)=F(x,z)=x$ and by conservativeness we also have
    $$
    F(x,F(y,z)) ~\in ~ \{F(x,y),F(x,z)\} ~=~ \{x\}.
    $$
    \item If $F(x,z)=z$, then by conservativeness and nondecreasing monotonicity we have $F(y,z)=z$ and hence $F(F(x,y),z)=F(x,z)=F(x,F(y,z))$.
    \end{itemize}
\item Suppose that $x\leq e\leq y\leq z$.  By Lemma~\ref{lemma:prel34}, we have $F(y,z)=\max\{y,z\}=z$.
    \begin{itemize}
    \item If $F(x,y)=x$, then we have $F(F(x,y),z)=F(x,z)=F(x,F(y,z))$.
    \item If $F(x,y)=y$, then by conservativeness and nondecreasing monotonicity we have $F(x,z)=z$ and hence $F(F(x,y),z)=F(y,z)=z=F(x,z)=F(x,F(y,z))$.
    \end{itemize}
\end{itemize}
This completes the proof of the proposition.
\end{proof}

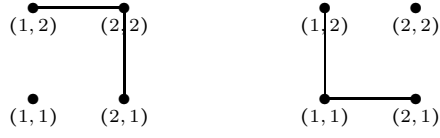
\begin{figure}[tbp]
\begin{center}
\begin{scriptsize}
\begin{picture}(2,2)
\multiput(0.5,0.5)(0,1){2}{\multiput(0,0)(1,0){2}{\circle*{0.12}}}%
\put(0.5,0.4){\makebox(0,0)[t]{$(1,1)$}}
\put(1.5,0.4){\makebox(0,0)[t]{$(2,1)$}}
\put(0.5,1.4){\makebox(0,0)[t]{$(1,2)$}}
\put(1.5,1.4){\makebox(0,0)[t]{$(2,2)$}}
\drawline[1](0.5,1.5)(1.5,1.5)(1.5,0.5)
\end{picture}
\hspace{0.10\textwidth}
\begin{picture}(2,2)
\multiput(0.5,0.5)(0,1){2}{\multiput(0,0)(1,0){2}{\circle*{0.12}}}%
\put(0.5,0.4){\makebox(0,0)[t]{$(1,1)$}}
\put(1.5,0.4){\makebox(0,0)[t]{$(2,1)$}}
\put(0.5,1.4){\makebox(0,0)[t]{$(1,2)$}}
\put(1.5,1.4){\makebox(0,0)[t]{$(2,2)$}}
\drawline[1](0.5,1.5)(0.5,0.5)(1.5,0.5)
\end{picture}
\end{scriptsize}
\caption{Possible operations when $n=2$}
\label{fig:5}
\end{center}
\end{figure}

\begin{figure}[tbp]
\begin{center}
\begin{scriptsize}
\begin{picture}(3,3)
\multiput(0.5,0.5)(0,1){3}{\multiput(0,0)(1,0){3}{\circle*{0.12}}}%
\put(0.5,0.4){\makebox(0,0)[t]{$(1,1)$}}
\put(1.5,0.4){\makebox(0,0)[t]{$(2,1)$}}
\put(2.5,0.4){\makebox(0,0)[t]{$(3,1)$}}
\put(0.5,1.4){\makebox(0,0)[t]{$(1,2)$}}
\put(1.5,1.4){\makebox(0,0)[t]{$(2,2)$}}
\put(2.5,1.4){\makebox(0,0)[t]{$(3,2)$}}
\put(0.5,2.4){\makebox(0,0)[t]{$(1,3)$}}
\put(1.5,2.4){\makebox(0,0)[t]{$(2,3)$}}
\put(2.5,2.4){\makebox(0,0)[t]{$(3,3)$}}
\drawline[1](0.5,2.5)(2.5,2.5)(2.5,0.5)
\drawline[1](0.5,1.5)(1.5,1.5)(1.5,0.5)
\end{picture}
\hspace{0.10\textwidth}
\begin{picture}(3,3)
\multiput(0.5,0.5)(0,1){3}{\multiput(0,0)(1,0){3}{\circle*{0.12}}}%
\put(0.5,0.4){\makebox(0,0)[t]{$(1,1)$}}
\put(1.5,0.4){\makebox(0,0)[t]{$(2,1)$}}
\put(2.5,0.4){\makebox(0,0)[t]{$(3,1)$}}
\put(0.5,1.4){\makebox(0,0)[t]{$(1,2)$}}
\put(1.5,1.4){\makebox(0,0)[t]{$(2,2)$}}
\put(2.5,1.4){\makebox(0,0)[t]{$(3,2)$}}
\put(0.5,2.4){\makebox(0,0)[t]{$(1,3)$}}
\put(1.5,2.4){\makebox(0,0)[t]{$(2,3)$}}
\put(2.5,2.4){\makebox(0,0)[t]{$(3,3)$}}
\drawline[1](0.5,2.5)(2.5,2.5)(2.5,0.5)
\drawline[1](0.5,1.5)(0.5,0.5)(1.5,0.5)
\end{picture}
\par\bigskip
\begin{picture}(3,3)
\multiput(0.5,0.5)(0,1){3}{\multiput(0,0)(1,0){3}{\circle*{0.12}}}%
\put(0.5,0.4){\makebox(0,0)[t]{$(1,1)$}}
\put(1.5,0.4){\makebox(0,0)[t]{$(2,1)$}}
\put(2.5,0.4){\makebox(0,0)[t]{$(3,1)$}}
\put(0.5,1.4){\makebox(0,0)[t]{$(1,2)$}}
\put(1.5,1.4){\makebox(0,0)[t]{$(2,2)$}}
\put(2.5,1.4){\makebox(0,0)[t]{$(3,2)$}}
\put(0.5,2.4){\makebox(0,0)[t]{$(1,3)$}}
\put(1.5,2.4){\makebox(0,0)[t]{$(2,3)$}}
\put(2.5,2.4){\makebox(0,0)[t]{$(3,3)$}}
\drawline[1](0.5,2.5)(0.5,0.5)(2.5,0.5)
\drawline[1](1.5,2.5)(2.5,2.5)(2.5,1.5)
\end{picture}
\hspace{0.10\textwidth}
\begin{picture}(3,3)
\multiput(0.5,0.5)(0,1){3}{\multiput(0,0)(1,0){3}{\circle*{0.12}}}%
\put(0.5,0.4){\makebox(0,0)[t]{$(1,1)$}}
\put(1.5,0.4){\makebox(0,0)[t]{$(2,1)$}}
\put(2.5,0.4){\makebox(0,0)[t]{$(3,1)$}}
\put(0.5,1.4){\makebox(0,0)[t]{$(1,2)$}}
\put(1.5,1.4){\makebox(0,0)[t]{$(2,2)$}}
\put(2.5,1.4){\makebox(0,0)[t]{$(3,2)$}}
\put(0.5,2.4){\makebox(0,0)[t]{$(1,3)$}}
\put(1.5,2.4){\makebox(0,0)[t]{$(2,3)$}}
\put(2.5,2.4){\makebox(0,0)[t]{$(3,3)$}}
\drawline[1](0.5,2.5)(0.5,0.5)(2.5,0.5)
\drawline[1](1.5,2.5)(1.5,1.5)(2.5,1.5)
\end{picture}
\end{scriptsize}
\caption{Possible operations when $n=3$}
\label{fig:6}
\end{center}
\end{figure}
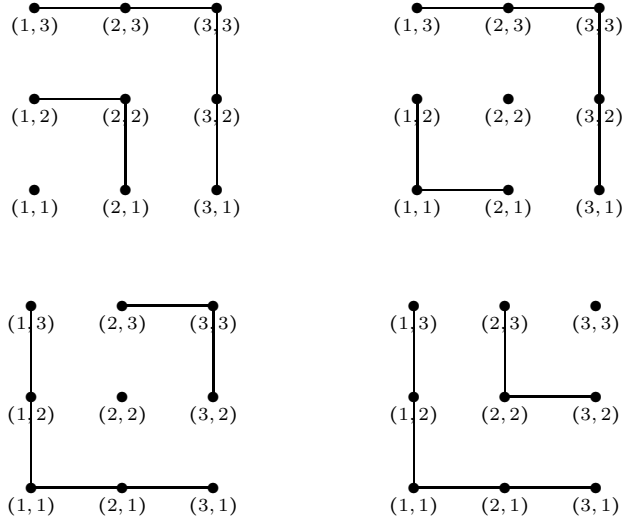

\begin{remark}
\begin{enumerate}
\item[(a)] The existence of a neutral element in Proposition~\ref{prop:main3} is no longer guaranteed if the chain is not finite. For instance, the real operation $F\colon [0,1]^2\to [0,1]$ defined by $F(x,y)=\max\{x,y\}$, if $x,y>0$, and $F(x,y)=0$, otherwise, is conservative, symmetric, and nondecreasing, but it does not have a neutral element.
\item[(b)] Associativity of $F$ in Proposition~\ref{prop:main3} can be established on any chain. This result was proved in the special case where the chain is the real unit interval $[0,1]$ in \cite[Proposition~2]{MarMayTor03} as a consequence of a sequence of three lemmas. Here we have provided a simpler proof based on the existence of a neutral element.
\item[(c)] We observe that conservativeness cannot be relaxed into idempotency in Proposition~\ref{prop:main3}. For instance the operation $F\colon L_3^2\to L_3$ whose contour plot is depicted in Figure~\ref{fig:2} is idempotent, symmetric, and nondecreasing, but one can show that it is not associative and it has no neutral element.
\item[(d)] We also observe that each of the conditions of Proposition~\ref{prop:main3} is necessary. Indeed, we give in Figure~\ref{fig:7} an operation that is conservative and symme{\-}tric but that is not nondecreasing. We also give in Figure~\ref{fig:8} an operation that is conservative and nondecreasing but not symmetric. Finally, we give in Figure~\ref{fig:9} an operation that is symmetric and nondecreasing but not conservative. None of these three operations is associative and none has a neutral element.
\end{enumerate}
\end{remark}

\begin{figure}[tbp]
\begin{center}
\begin{small}
\begin{picture}(3,3)
\multiput(0.5,0.5)(0,1){3}{\multiput(0,0)(1,0){3}{\circle*{0.12}}}%
\drawline[1](2.5,1.5)(2.5,2.5)(1.5,2.5)
\drawline[1](0.5,1.5)(1.5,1.5)(1.5,0.5)
\put(1.5,0.5){\oval(2,0.5)[b]}\put(0.5,1.5){\oval(0.5,2)[l]}
\put(0.7,0.5){\makebox(0,0)[l]{$1$}}
\put(1.7,1.5){\makebox(0,0)[l]{$2$}}
\put(2.7,2.5){\makebox(0,0)[l]{$3$}}
\end{picture}
\end{small}
\caption{An operation that is not nondecreasing}
\label{fig:7}
\end{center}
\end{figure}
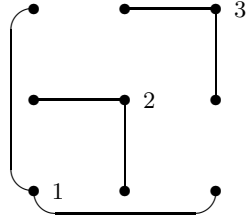

\begin{figure}[tbp]
\begin{minipage}{0.49\textwidth}
\begin{center}
\begin{small}
\begin{picture}(3,3)
\multiput(0.5,0.5)(0,1){3}{\multiput(0,0)(1,0){3}{\circle*{0.12}}}%
\drawline[1](0.5,0.5)(0.5,2.5)
\drawline[1](1.5,0.5)(1.5,1.5)
\drawline[1](1.5,2.5)(2.5,2.5)(2.5,0.5)
\put(0.7,0.5){\makebox(0,0)[l]{$1$}}
\put(1.7,1.5){\makebox(0,0)[l]{$2$}}
\put(2.7,2.5){\makebox(0,0)[l]{$3$}}
\end{picture}
\end{small}
\caption{An operation that is not symmetric}
\label{fig:8}
\end{center}
\end{minipage}
\hfill
\begin{minipage}{0.49\textwidth}
\begin{center}
\begin{small}
\begin{picture}(3,3)
\multiput(0.5,0.5)(0,1){3}{\multiput(0,0)(1,0){3}{\circle*{0.12}}}%
\drawline[1](0.5,2.5)(0.5,0.5)(2.5,0.5)
\drawline[1](0.5,1.5)(1.5,1.5)(1.5,0.5)
\drawline[1](1.5,2.5)(2.5,2.5)(2.5,1.5)
\put(1.7,1.5){\makebox(0,0)[l]{$1$}}
\put(2.7,2.5){\makebox(0,0)[l]{$3$}}
\end{picture}
\end{small}
\caption{An operation that is not conservative}
\label{fig:9}
\end{center}
\end{minipage}
\end{figure}

\begin{theorem}\label{thm:main}
An operation $F\colon L_n^2\to L_n$ is conservative, symmetric, and nondecreasing if and only if it is an idempotent discrete uninorm.
\end{theorem}

\begin{proof}
(Necessity) The result immediately follows from Proposition~\ref{prop:main3} and the fact that any conservative operation is idempotent (see Remark~\ref{rem:ConsJ}).

(Sufficiency) By definition, any idempotent discrete uninorm is symmetric and nondecreasing. It is also conservative by Theorem~\ref{thm:Deb1} (for an alternative proof see Remark~\ref{rem:ccons}).
\end{proof}

\begin{corollary}
Let $F\colon L_n^2\to L_n$ be an operation that is associative, symmetric, nondecreasing, and has a neutral element. Then $F$ is idempotent if and only if it is conservative.
\end{corollary}

%

The following result gives the exact number of idempotent discrete uninorms on $L_n$. A first proof that essentially relies on Theorem~\ref{thm:Deb1} can be found in \cite[Theorem~4]{DeBFodRuiTor09}. An alternative proof is also available in \cite[Remark~1]{Mes16}. Here we provide a new proof based on Theorem~\ref{thm:main} only.

\begin{theorem}\label{thm:main2n}
There are exactly $2^{n-1}$ idempotent discrete uninorms on $L_n$.
\end{theorem}

\begin{proof}
By Theorem~\ref{thm:main} it is enough to count the number $c_n$ of operations on $L_n$ that are conservative, symmetric, and nondecreasing. As already observed in the proof of Proposition~\ref{prop:main3} we can see by inspection that $c_1=1$, $c_2=2$, and $c_3=4$ (see Figures~\ref{fig:5} and \ref{fig:6}). Suppose now that $c_{n-1}=2^{n-2}$ for some integer $n\geq 3$ and let us prove that $c_n=2^{n-1}$. Let $F\colon L_n^2\to L_n$ be an arbitrary idempotent discrete uninorm. By conservativeness and symmetry we then have $F(1,n)=F(n,1)\in\{1,n\}$. Suppose first that $F(1,n)=F(n,1)=1$. By nondecreasing monotonicity, we also have $F(1,x)=F(x,1)=1$ for all $x\in L_n$. Consider the subchain $L'=L_n\setminus\{1\}$. Clearly, the operation $F'=F|_{L'}$ is conservative, symmetric, and nondecreasing and there are $c_{n-1}=2^{n-2}$ possible such operations (see Figure~\ref{fig:10}, on left). We arrive at the same conclusion if $F(1,n)=F(n,1)=n$ (see Figure~\ref{fig:10}, on right). In total, we then have $c_n=2^{n-2}+2^{n-2}=2^{n-1}$.
\end{proof}

\begin{figure}[tbp]
\begin{center}
\begin{picture}(3,3)
\put(0.5,0.5){\circle*{0.12}}\put(0.5,1.0){\circle*{0.12}}\put(0.5,2.5){\circle*{0.12}}
\put(1.0,0.5){\circle*{0.12}}\put(2.5,0.5){\circle*{0.12}}
\drawline[1](0.5,2.5)(0.5,0.5)(2.5,0.5)
\put(1,1){\dashbox{0.06}(1.5,1.5){$F|_{L'}$}}
\end{picture}
\hspace{0.10\textwidth}
\begin{picture}(3,3)
\put(2.5,2.5){\circle*{0.12}}\put(2.0,2.5){\circle*{0.12}}\put(0.5,2.5){\circle*{0.12}}
\put(2.5,2.0){\circle*{0.12}}\put(2.5,0.5){\circle*{0.12}}
\drawline[1](0.5,2.5)(2.5,2.5)(2.5,0.5)
\put(0.5,0.5){\dashbox{0.06}(1.5,1.5){$F|_{L'}$}}
\end{picture}
\caption{Illustration of the proof of Theorem~\ref{thm:main2n}}
\label{fig:10}
\end{center}
\end{figure}
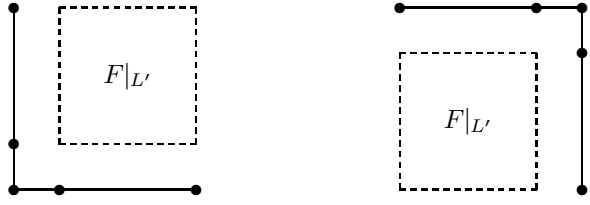

Theorems~\ref{thm:main} and \ref{thm:main2n} together enable us to provide the following graphical characterization of the idempotent discrete uninorms in terms of their contour plots.

\begin{theorem}\label{thm:GC}
The following algorithm outputs the contour plot of an arbitrary idempotent discrete uninorm $F\colon L_n^2\to L_n$.
\begin{enumerate}
  \item[Step 1.] Choose the neutral element $a_1\in L_n$ and set $C_1=\{a_1\}$. The point $(a_1,a_1)$ is necessarily isolated with value $a_1$
  \item[Step 2.] For $k=2,\ldots,n$
  \begin{enumerate}
  \item[1.] Pick a closest element $a_k$ to $C_{k-1}$ in $L_n\setminus C_{k-1}$
  \item[2.] Set $C_k=\{a_k\}\cup C_{k-1}$
  \item[3.] Connect all the points in $C_k^2\setminus C_{k-1}^2$ with common value $a_k$
  \end{enumerate}
\end{enumerate}
\end{theorem}

\begin{proof}
For every $e=a_1\in L_n$ chosen in Step 1, denote by $c_n(e)$ the number of possible operations constructed in Step 2. We show by induction on $n$ that $c_n(e)={n-1\choose e-1}$. We clearly have $c_1(1)=1$. It is also easy to see that $c_2(1)=c_2(2)=1$ (see Figure~\ref{fig:5}). Suppose now that $c_{n-1}(e)={n-2\choose e-1}$ for some integer $n\geq 3$ and let us compute $c_n(e)$ for any $e\in L_n$. We clearly have $c_n(1)=c_n(n)=1$ so we can assume that $1<e<n$. It is then easy to see that if $a_2=e-1$ (resp.\ $a_2=e+1$), then the number of possible operations constructed in Step 2 is $c_{n-1}(e-1)={n-2\choose e-2}$ (resp.\ $c_{n-1}(e)={n-2\choose e-1}$). In total we obtain $c_n(e)={n-2\choose e-2}+{n-2\choose e-1}={n-1\choose e-1}$.

Let us now show that the algorithm enables us to generate all the idempotent discrete uninorms on $L_n$. On the one hand, we clearly see that the algorithm enables us to construct $\sum_{e=1}^nc_n(e)=2^{n-1}$ distinct operations. On the other hand, all these operations are clearly conservative, symmetric, and nondecreasing, and hence they are idempotent discrete uninorms by Theorem~\ref{thm:main}. We then conclude the proof by Theorem~\ref{thm:main2n}.
\end{proof}

Figure~\ref{fig:11} gives two possible idempotent discrete uninorms on $L_4$. All the possible idempotent discrete uninorms on $L_2$ and $L_3$ are given in Figures~\ref{fig:5} and \ref{fig:6}, respectively.

\begin{figure}[tbp]
\begin{center}
\begin{small}
\begin{picture}(4,4)
\multiput(0.5,0.5)(0,1){4}{\multiput(0,0)(1,0){4}{\circle*{0.12}}}%
\drawline[1](0.5,3.5)(0.5,0.5)(3.5,0.5)
\drawline[1](1.5,2.5)(1.5,1.5)(2.5,1.5)
\drawline[1](1.5,3.5)(3.5,3.5)(3.5,1.5)
\put(3.7,0.5){\makebox(0,0)[l]{$1$}}
\put(2.7,1.5){\makebox(0,0)[l]{$2$}}
\put(2.3,2.5){\makebox(0,0)[r]{$3$}}
\put(1.3,3.5){\makebox(0,0)[r]{$4$}}
\end{picture}
\hspace{0.10\textwidth}
\begin{picture}(4,4)
\multiput(0.5,0.5)(0,1){4}{\multiput(0,0)(1,0){4}{\circle*{0.12}}}%
\drawline[1](1.5,0.5)(1.5,1.5)(0.5,1.5)
\drawline[1](2.5,0.5)(2.5,2.5)(0.5,2.5)
\drawline[1](3.5,0.5)(3.5,3.5)(0.5,3.5)
\put(0.3,0.5){\makebox(0,0)[r]{$1$}}
\put(0.3,1.5){\makebox(0,0)[r]{$2$}}
\put(0.3,2.5){\makebox(0,0)[r]{$3$}}
\put(0.3,3.5){\makebox(0,0)[r]{$4$}}
\end{picture}
\end{small}
\caption{Two possible idempotent discrete uninorms}
\label{fig:11}
\end{center}
\end{figure}
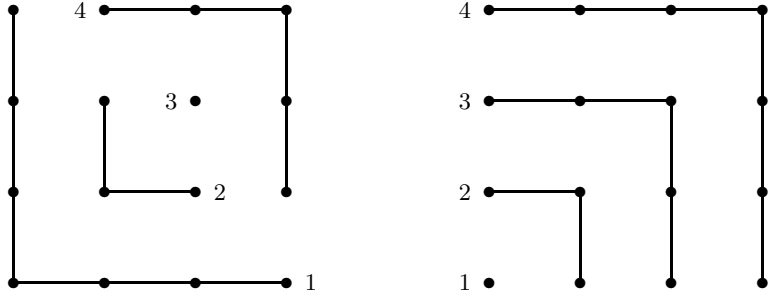

The next theorem provides an algebraic translation of the graphical characterization given in Theorem~\ref{thm:GC}. Let us first recall the concept of single-peaked ordering, as introduced in social choice theory by Black~\cite{Bla48,Bla87} (see \cite{BerPer06} for a more recent reference).

\begin{definition}
A linear ordering $\preceq$ on $L_n$ is said to be \emph{single-peaked} (with respect to the ordering $\leq$) if for any $a,b,c\in L_n$ such that $a<b<c$ we have $b\prec a$ or $b\prec c$.
\end{definition}

Thus defined, single-peakedness is equivalent to saying that from among three distinct elements of $L_n$ the centrist one is never ranked last by $\preceq$. It is then easy to see that any single-peaked linear ordering $a_1\prec\cdots\prec a_n$ on $L_n$ can be constructed as follows.
\begin{enumerate}
\item[1.] Choose $a_1\in L_n$.
\item[2.] For $k=2,\ldots,n$, choose for $a_k$ a closest element to the set $C_{k-1}$ in $L_n\setminus C_{k-1}$, where $C_k=\{a_1,\ldots,a_k\}$.
\end{enumerate}
One can easily show (see, e.g., \cite{BerPer06}) by induction on $n$ that there are exactly $2^{n-1}$ single-peaked linear orderings on $L_n$ (first observe that the last-ranked element $a_n$ of any single-peaked linear ordering $a_1\prec\cdots\prec a_n$ is always $1$ or $n$).

\begin{remark}
The single-peakedness property of a linear ordering $\preceq$ on $L_n$ can be easily checked by plotting a discrete function $f_{\preceq}$ in a rectangular coordinate system in the following way. Place the reference chain $(L_n,\leq)$ on the horizontal axis and the reversed version of the chain $(L_n,\preceq)$ on the vertical axis. The graph of $f_{\preceq}$ is obtained by joining the points $(1,1),\ldots,(n,n)$ by line segments. We then see that the linear ordering $\preceq$ is single-peaked if and only if $f_{\preceq}$ has only one local maximum. Figure~\ref{fig:bb} gives the functions $f_{\preceq}$ and $f_{\preceq'}$ corresponding to the linear orderings $2\prec 3\prec 4\prec 1\prec 5$ and $5\prec' 2\prec' 1\prec' 3\prec' 4$, respectively, on $L_5$. We see that $\preceq$ is single-peaked since $f_{\preceq}$ has only one local maximum while $\preceq'$ is not single-peaked since $f_{\preceq'}$ has two local maxima.
\end{remark}

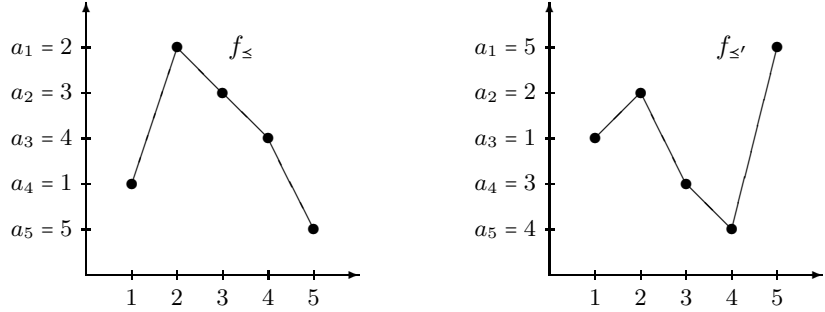
\begin{figure}[tbp]
\begin{center}
\begin{small}
\begin{picture}(5,3.75)
\put(1.5,0.5){\vector(1,0){3}}\put(1.5,0.5){\vector(0,1){3}}
\multiput(2,0.45)(0.5,0){5}{\line(0,1){0.1}}%
\multiput(1.45,1)(0,0.5){5}{\line(1,0){0.1}}%
\put(2,0.25){\makebox(0,0){$1$}}
\put(2.5,0.25){\makebox(0,0){$2$}}
\put(3,0.25){\makebox(0,0){$3$}}
\put(3.5,0.25){\makebox(0,0){$4$}}
\put(4,0.25){\makebox(0,0){$5$}}
\put(1.35,1){\makebox(0,0)[r]{$a_5=5$}}
\put(1.35,1.5){\makebox(0,0)[r]{$a_4=1$}}
\put(1.35,2){\makebox(0,0)[r]{$a_3=4$}}
\put(1.35,2.5){\makebox(0,0)[r]{$a_2=3$}}
\put(1.35,3){\makebox(0,0)[r]{$a_1=2$}}
\drawline[1](2,1.5)(2.5,3)(3.5,2)(4,1)
\put(2,1.5){\circle*{0.12}}
\put(2.5,3){\circle*{0.12}}
\put(3,2.5){\circle*{0.12}}
\put(3.5,2){\circle*{0.12}}
\put(4,1){\circle*{0.12}}
\put(3.2,3){\makebox(0,0){\begin{normalsize}$f_{\preceq}$\end{normalsize}}}
\end{picture}
\begin{picture}(5,3.75)
\put(1.5,0.5){\vector(1,0){3}}\put(1.5,0.5){\vector(0,1){3}}
\multiput(2,0.45)(0.5,0){5}{\line(0,1){0.1}}%
\multiput(1.45,1)(0,0.5){5}{\line(1,0){0.1}}%
\put(2,0.25){\makebox(0,0){$1$}}
\put(2.5,0.25){\makebox(0,0){$2$}}
\put(3,0.25){\makebox(0,0){$3$}}
\put(3.5,0.25){\makebox(0,0){$4$}}
\put(4,0.25){\makebox(0,0){$5$}}
\put(1.35,1){\makebox(0,0)[r]{$a_5=4$}}
\put(1.35,1.5){\makebox(0,0)[r]{$a_4=3$}}
\put(1.35,2){\makebox(0,0)[r]{$a_3=1$}}
\put(1.35,2.5){\makebox(0,0)[r]{$a_2=2$}}
\put(1.35,3){\makebox(0,0)[r]{$a_1=5$}}
\drawline[1](2,2)(2.5,2.5)(3,1.5)(3.5,1)(4,3)
\put(2,2){\circle*{0.12}}
\put(2.5,2.5){\circle*{0.12}}
\put(3,1.5){\circle*{0.12}}
\put(3.5,1){\circle*{0.12}}
\put(4,3){\circle*{0.12}}
\put(3.5,3){\makebox(0,0){\begin{normalsize}$f_{\preceq'}$\end{normalsize}}}
\end{picture}
\end{small}
\caption{$\preceq$ is single-peaked (left) while $\preceq'$ is not (right)}
\label{fig:bb}
\end{center}
\end{figure}


\begin{theorem}\label{thm:QoB}
An operation $F\colon L_n^2\to L_n$ is an idempotent discrete uninorm if and only if there exists a single-peaked linear ordering $\preceq$ such that $F=\max_{\preceq}$, where $\max_{\preceq}$ denotes the maximum operation with respect to $\preceq$.
\end{theorem}

\begin{proof}
Clearly, the algorithm given in Theorem~\ref{thm:GC} provides a single-peaked linear ordering $a_1\prec\cdots\prec a_n$ on $L_n$ together with the operation $F\colon L_n^2\to L_n$ defined as
$$
F(x,y) ~=~
\begin{cases}
a_1, & \text{if $a_1\in\{x,y\}$ and $a_2,\ldots,a_n\notin\{x,y\}$},\\
a_2, & \text{if $a_2\in\{x,y\}$ and $a_3,\ldots,a_n\notin\{x,y\}$},\\
& \vdots\\
a_n, & \text{if $a_n\in\{x,y\}$},
\end{cases}
$$
for all $x,y\in L_n$, or equivalently, $F=\max_{\preceq}$.
\end{proof}

We now provide a direct proof of Theorem~\ref{thm:QoB} without using the algorithm given in Theorem~\ref{thm:GC}.

\begin{proof}[Alternative proof of Theorem~\ref{thm:QoB}]
Since there are as many idempotent discrete uninorms on $L_n$ as single-peaked linear orderings on $L_n$, namely $2^{n-1}$, we simply need to show that for any single-peaked linear ordering $\preceq$ on $L_n$, the operation $F=\max_{\preceq}$ is an idempotent discrete uninorm, or equivalently, it is conservative, symmetric, and nondecreasing (by Theorem~\ref{thm:main}). Only nondecreasing monotonicity needs to be verified. Suppose on the contrary that there exist $a_1,a_2,b\in L_n$, with $a_1<a_2$, such that $F(a_1,b)>F(a_2,b)$, that is,
$$
\textstyle{\max_{\preceq}\{a_1,b\} ~>~ \max_{\preceq}\{a_2,b\}.}
$$
We only have two relevant cases to consider.
\begin{itemize}
\item If $a_1\preceq b\preceq a_2$, then we obtain $a_1<a_2<b$.
\item If $a_2\preceq b\preceq a_1$, then we obtain $b<a_1<a_2$.
\end{itemize}
We immediately reach a contradiction since single-peakedness is violated in each of these two cases.
\end{proof}

\begin{remark}
\begin{enumerate}
\item[(a)] The single-peakedness property can be easily expressed in terms of permutations on $L_n$ (see, e.g., \cite{Fou01}). Indeed, with any linear ordering relation $\preceq$ on $L_n$ we can associate a unique permutation $\sigma$ on $L_n$ defined as
$$
\sigma(a)\preceq\sigma(b)\quad\Leftrightarrow\quad a\leq b{\,},\qquad a,b\in L_n.
$$
Thus, the ordering $\preceq$ (or equivalently, the permutation $\sigma$) is \emph{single-peaked} if for any $a,b,c\in L_n$ such that $a<b<c$ we have $\sigma^{-1}(b)<\sigma^{-1}(a)$ or $\sigma^{-1}(b)<\sigma^{-1}(c)$. Using this terminology, the maximum operation with respect to $\preceq$ is the conjugate operation $\max_{\sigma}\colon L_n^2\to L_n$ defined as
$$
\textstyle{\max_{\sigma}(x,y) ~=~ \sigma(\max\{\sigma^{-1}(x),\sigma^{-1}(y)\})},\qquad x,y\in L_n.
$$
Indeed, for any $x,y\in L_n$, we have $\max_{\preceq}(x,y)=y$ if and only if $x\preceq y$ if and only if $\sigma^{-1}(x)\leq\sigma^{-1}(y)$ if and only if $\max_{\sigma}(x,y)=y$.
\item[(b)] It is clear that in the definition of single-peakedness the condition ``$a<b<c$'' can be equivalently replaced with ``$c<b<a$''. However, if we replace the condition ``$b\prec a$ or $b\prec c$'' with ``$a\prec b$ or $c\prec b$'', then we obtain a dual version of single-peakedness that expresses that from among three distinct elements of $L_n$ the centrist one is never ranked first by $\preceq$. Thus, considering the single-peakedness property instead of its dual version is simply a matter of convention. Note that if we use the dual version, then the maximum operation in Theorem~\ref{thm:QoB} has to be replaced with the minimum operation.
\end{enumerate}
\end{remark}

We end this section by giving a graphical test for checking whether a conservative operation is associative.

\begin{proposition}\label{prop:testCA}
Let $F\colon X^2\to X$ be a conservative operation. Then the following assertions are equivalent.
\begin{enumerate}
\item[(i)] $F$ is not associative.
\item[(ii)] There exist pairwise distinct $a,b,c\in X$ such that $F(a,b)$, $F(a,c)$, $F(b,c)$ are pairwise distinct.
\item[(iii)] There exists a rectangle in $X^2$ such that one of its vertices is on $\Delta_X$ and the three remaining vertices are in $X^2\setminus\Delta_X$ and pairwise disconnected.
\end{enumerate}
\end{proposition}

\begin{proof}
The equivalence between (i) and (ii) was shown in \cite[Lemma~1 and Corollary~1]{MarMayTor03}. The equivalence between (ii) and (iii) is immediate. Just consider the rectangle constructed on the vertices $(a,c)$, $(b,c)$, $(b,b)$, $(a,b)$ for pairwise distinct $a,b,c\in X$.
\end{proof}

\begin{example}
As an application of Proposition~\ref{prop:testCA}, let us consider the operation $F\colon L_3^2\to L_3$ defined in Figure~\ref{fig:7}. We can see that this operation is not associative because the rectangle constructed on the vertices $(2,2)$, $(3,2)$, $(3,1)$, $(2,1)$ has a vertex on $\Delta_X$ and the three other vertices are in $X^2\setminus\Delta_X$ and pairwise disconnected. To give a second example, Figure~\ref{fig:12} (left) represents an operation that is conservative, associative, and not nondecreasing. Associativity can be verified by considering the six rectangles shown in Figure~\ref{fig:12} (right).
\end{example}

\begin{figure}[tbp]
\begin{center}
\begin{small}
\begin{picture}(3,3)
\multiput(0.5,0.5)(0,1){3}{\multiput(0,0)(1,0){3}{\circle*{0.12}}}%
\drawline[1](1.5,2.5)(1.5,0.5)
\drawline[1](0.5,1.5)(2.5,1.5)
\put(1.5,2.5){\oval(2,0.5)[t]}\put(2.5,1.5){\oval(0.5,2)[r]}
\put(0.6,0.6){\makebox(0,0)[lb]{$1$}}
\put(1.6,1.6){\makebox(0,0)[lb]{$2$}}
\put(2.6,2.6){\makebox(0,0)[lb]{$3$}}
\end{picture}
\hspace{0.10\textwidth}
\begin{picture}(5,3)
\multiput(0.5,0.5)(0,1.4){2}{\multiput(0,0)(1.6,0){3}{\multiput(0,0)(0,0.4){3}{\multiput(0,0)(0.4,0){3}{\circle*{0.075}}}}}%
\put(0.5,2.3){\framebox(0.4,0.4)}
\put(2.5,1.9){\framebox(0.4,0.4)}
\put(3.7,1.9){\framebox(0.4,0.8)}
\put(0.9,0.5){\framebox(0.4,0.8)}
\put(2.1,0.9){\framebox(0.8,0.4)}
\put(3.7,0.5){\framebox(0.8,0.4)}
\end{picture}
\end{small}
\caption{An associative operation (left) and the six rectangles to be checked (right)}
\label{fig:12}
\end{center}
\end{figure}
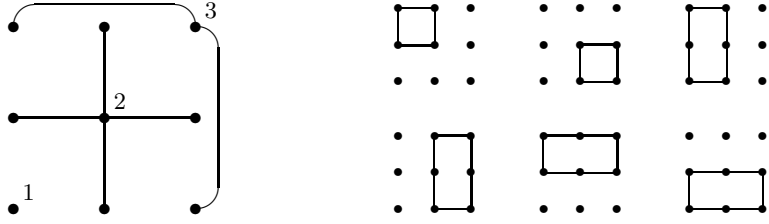

\begin{proposition}\label{prop:rec8n}
If $X=L_n$, then there are exactly $n(n-1)(n-2)$ rectangles in $X^2$ satisfying the following property: one of the vertices is on $\Delta_X$ and the three remaining vertices are in $X^2\setminus\Delta_X$.
\end{proposition}

\begin{proof}
As observed in the proof of Proposition~\ref{prop:testCA}, each of these rectangles is constructed on the vertices $(a,c)$, $(b,c)$, $(b,b)$, $(a,b)$ for some pairwise distinct $a,b,c\in L_n$. We then immediately conclude the proof by observing that there are exactly $n(n-1)(n-2)$ triplets $(a,b,c)\in L_n^3$ such that $a,b,c$ are pairwise distinct.
\end{proof}

\begin{remark}\label{rem:3sf13}
In the special case where the operation $F$ is symmetric, the number of rectangles to be checked reduces to $\frac{1}{6}n(n-1)(n-2)={n\choose 3}$, which represents the number of $3$-element subsets of $L_n$. As an example, the associativity of the conservative and symmetric operation given in Figure~\ref{fig:12} can actually be verified by considering \emph{any} of the six rectangles.
\end{remark}

\section{Bisymmetric operations}

In this section we provide a characterization of the class of discrete uninorms in terms of the bisymmetry (or mediality) property. From this result we immediately derive a new characterization of the class of idempotent discrete uninorms.

\begin{definition}
An operation $F\colon X^2\to X$ is said to be \emph{bisymmetric} if
$$
F(F(x,y),F(u,v))~=~F(F(x,u),F(y,v))
$$
for all $x,y,u,v\in X$.
\end{definition}

\begin{lemma}\label{lemma:bis}
Let $F\colon X^2\to X$ be an operation. Then the following assertions hold.
\begin{enumerate}
\item[(a)] If $F$ is bisymmetric and has a neutral element, then it is associative and symmetric.
\item[(b)] If $F$ is associative and symmetric, then it is bisymmetric.
\item[(c)] If $F$ is bisymmetric and conservative, then it is associative.
\end{enumerate}
\end{lemma}

\begin{proof}
(a) Let $x,y,z\in X$ and let $e$ be the neutral element of $F$. Then we have
$$
F(F(x,y),z) ~=~ F(F(x,y),F(e,z)) ~=~ F(F(x,e),F(y,z)) ~=~ F(x,F(y,z))
$$
and
$$
F(x,y) ~=~ F(F(e,x),F(y,e)) ~=~ F(F(e,y),F(x,e)) ~=~ F(y,x).
$$

(b) The result immediately follows from the identities
$$
F(F(x,y),F(u,v)) ~=~ F(x,F(y,F(u,v))) ~=~ F(x,F(F(y,u),v))),
$$
which show that $F(F(x,y),F(u,v))$ is symmetric on $y$ and $u$.

(c) Let $x,y,z\in X$. Then we have $F(x,z)\in\{x,z\}$. If $F(x,z)=x$, then
$$
F(F(x,y),z) ~=~ F(F(x,y),F(z,z)) ~=~ F(F(x,z),F(y,z)) ~=~ F(x,F(y,z)).
$$
If $F(x,z)=z$, we have
$$
F(F(x,y),z) ~=~ F(F(x,y),F(x,z)) ~=~ F(F(x,x),F(y,z)) ~=~ F(x,F(y,z)).
$$
This shows that $F$ is associative.
\end{proof}

\begin{remark}\label{rem:AM}
\begin{enumerate}
\item[(a)] Using Lemma~\ref{lemma:bis} we immediately see that if an operation is conservative and symmetric, then it is associative if and only if it is bisymmetric. Combining this observation with Proposition~\ref{prop:testCA} and Remark~\ref{rem:3sf13} provides a test for bisymmetry under conservativeness and symmetry.
\item[(b)] We observe that the conjunction of bisymmetry and symmetry implies neither associativity nor the existence of a neutral element (take for instance the arithmetic mean over the reals).
\item[(c)] Also, the conjunction of conservativeness and associativity does not imply bisymmetry, even in the presence of a neutral element. We give in Figure~\ref{fig:13} an operation that is conservative, associative, and has $1$ as the neutral element. However it is not bisymmetric since $F(F(1,2),F(3,2))\neq F(F(1,3),F(2,2))$.
\item[(d)] The first two assertions of Lemma~\ref{lemma:bis} were already known in the literature. Assertion (a) was proved for finite chains in \cite[Lemma~3.3]{MasMonTor03} and \cite[Lemma~3]{SuLiuPed17}. Assertion (b) was proved for the real unit interval $[0,1]$ in \cite[p.~180]{San05}. These proofs are purely algebraic and work for any nonempty set $X$. To our knowledge, assertion (c) was previously unknown.
\end{enumerate}
\end{remark}

\begin{figure}[tbp]
\begin{center}
\begin{small}
\begin{picture}(3,3)
\multiput(0.5,0.5)(0,1){3}{\multiput(0,0)(1,0){3}{\circle*{0.12}}}%
\drawline[1](0.5,1.5)(1.5,1.5)
\drawline[1](1.5,2.5)(1.5,0.5)
\drawline[1](2.5,2.5)(2.5,0.5)
\put(1.5,2.5){\oval(2,0.5)[t]}
\put(0.7,0.5){\makebox(0,0)[l]{$1$}}
\put(1.7,1.5){\makebox(0,0)[l]{$2$}}
\put(2.7,2.5){\makebox(0,0)[l]{$3$}}
\end{picture}
\end{small}
\caption{An associative operation that is not bisymmetric}
\label{fig:13}
\end{center}
\end{figure}
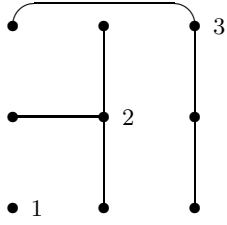


Using Lemma~\ref{lemma:bis}, we immediately derive the following two characterizations. The first one was already established in \cite[Corollary 1]{SuLiuPed17} in the special setting of aggregation operators (i.e., operations $F\colon L_n^2\to L_n$ that are nondecreasing and that satisfy $F(1,1)=1$ and $F(n,n)=n$).

\begin{theorem}\label{thm:mainb}
An operation $F\colon L_n^2\to L_n$ is bisymmetric, nondecreasing, and has a neutral element if and only if it is a discrete uninorm.
\end{theorem}

\begin{corollary}\label{cor:mainb}
An operation $F\colon L_n^2\to L_n$ is idempotent (or conservative), bisymmetric, nondecreasing, and has a neutral element if and only if it is an idempotent discrete uninorm.
\end{corollary}

Since Theorem~\ref{thm:mainb} and Corollary~\ref{cor:mainb} are derived from Lemma~\ref{lemma:bis} only, we immediately see that these results still hold on any chain, or even on any ordered set, provided the discrete uninorm is replaced with a uninorm (i.e., a binary operation that is associative, symmetric, nondecreasing, and has a neutral element).

\begin{remark}
We observe that bisymmetry is necessary in Corollary~\ref{cor:mainb}. For instance the operation $F\colon L_4^2\to L_4$ whose contour plot is depicted in Figure~\ref{fig:14} is conservative, nondecreasing, and has a neutral element. However, it is neither associative nor symmetric.
\end{remark}

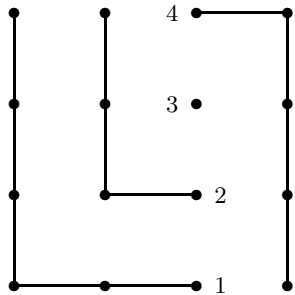
\begin{figure}[tbp]
\begin{center}
\begin{small}
\begin{picture}(4,4)
\multiput(0.5,0.5)(0,1){4}{\multiput(0,0)(1,0){4}{\circle*{0.12}}}%
\drawline[1](0.5,3.5)(0.5,0.5)(2.5,0.5)
\drawline[1](1.5,3.5)(1.5,1.5)(2.5,1.5)
\drawline[1](2.5,3.5)(3.5,3.5)(3.5,0.5)
\put(2.7,0.5){\makebox(0,0)[l]{$1$}}
\put(2.7,1.5){\makebox(0,0)[l]{$2$}}
\put(2.3,2.5){\makebox(0,0)[r]{$3$}}
\put(2.3,3.5){\makebox(0,0)[r]{$4$}}
\end{picture}
\end{small}
\caption{An operation that is neither associative nor symmetric}
\label{fig:14}
\end{center}
\end{figure}

\section{Concluding remarks}

In this paper we established three main characterizations of the class of idempotent discrete uninorms on finite chains. Two of them are of axiomatic nature (see Theorem~\ref{thm:main} and Corollary~\ref{cor:mainb}) while the third one is of graphical nature (see Theorem~\ref{thm:GC} and its algebraic version Theorem~\ref{thm:QoB}). These axiomatic characterizations are essentially based on conservativeness, which is a rather strong property that can be easily justified in some contexts (see Remark~\ref{rem:ConsJ}). The graphical characterization is a rather surprising result that shows that the idempotent discrete uninorms on a given finite chain can be very easily generated from a graphical viewpoint. This result contrasts with the rather intricate descriptive characterization given in Theorem~\ref{thm:Deb1}.

In this work we put a particular emphasis on the graphical properties of operations by looking into their contour plots. In particular, we have presented graphical tests to verify whether an operation:
\begin{itemize}
\item is conservative (Proposition~\ref{prop:Tcons}),
\item has a neutral element (Proposition~\ref{prop:Te3}),
\item is an idempotent discrete uninorm (Theorem~\ref{thm:GC}),
\end{itemize}
and whether a conservative operation:
\begin{itemize}
\item has a neutral element (Proposition~\ref{prop:ee}),
\item is associative (Proposition~\ref{prop:testCA}).
\end{itemize}

In view of these results, some questions emerge naturally, and we end this paper by listing a few below.

\begin{enumerate}
\item[(a)] Enumerate and/or generate all the conservative (resp.\ conservative and associative, etc.) operations on a finite chain.
\item[(b)] Provide a graphical test for checking whether a conservative operation is bisymmetric. Using Remark~\ref{rem:AM}(a) we already have such a test for symmetric operations.
\item[(c)] Knowing that $F\colon L_n^2\to L_n$ is bisymmetric and symmetric, does it imply that it is associative and that it has a neutral element? We know from Remark~\ref{rem:AM}(b) that this is not true on an arbitrary chain.
\end{enumerate}

\section*{Acknowledgments}

The authors would like to thank Gergely Kiss for fruitful discussions and valuable remarks. This research is partly supported by the internal research project R-AGR-0500 of the University of Luxembourg.

\end{document}